\documentclass[11pt,a4paper]{amsart}
\usepackage[foot]{amsaddr}
\usepackage{texmac}
\usepackage{enumitem}
\usepackage[all]{xy}
\xyoption{web}

\operators{Stab,Prim,Inn,Out,ss,B,G,R,K,M,Inf,Def,Imprim}

\bbsymbols{b}{C,F,Z,Q}
\calsymbols{c}{H,U,X,Y,P,M,N,D,F,G,K,I,Q,S}
\fraksymbols{f}{p,q,a}

\def\dsum_#1_#2{\sum_{{#1}\atop {#2}}}
\def\rar{\rightarrow}
\DeclareMathOperator\rZ{Z}

\begin{document}

\title[A note on Green functors with inflation]
{A note on Green functors with inflation}
\author{Alex Bartel}
\address{Mathematics Institute, Zeeman Building, University of Warwick,
Coventry CV4 7AL, UK}
\author{Matthew Spencer}
\email{a.bartel@warwick.ac.uk, M.J.Spencer@warwick.ac.uk}
\llap{.\hskip 10cm} \vskip -0.8cm
\maketitle

\begin{abstract}
This note is motivated by the problem to understand, given a commutative ring $F$,
which $G$-sets $X$, $Y$ give rise to isomorphic $F[G]$-representations
$F[X]\cong F[Y]$. A typical step in such investigations is an argument that
uses induction theorems to give very general sufficient conditions for all
such relations to come from proper subquotients of $G$. In the present paper
we axiomatise the situation, and prove such a result in the generality of
Mackey functors and Green functors with inflation.
Our result includes, as special cases, a result of Deligne on monomial
relations, a result of the first author and Tim Dokchitser on Brauer
relations in characteristic 0, and a new result on Brauer relations
in characteristic $p>0$. We will need the new result in a forthcoming paper
on Brauer relations in positive characteristic.
\end{abstract}



\section{Introduction}

The Burnside ring $\B(G)$ of a finite group $G$ is, as a group, the free abelian
group on the set of isomorphism classes of transitive $G$-sets. Any transitive
$G$-set is isomorphic to a set of cosets $G/H$ for some $H\leq G$, so we may
write elements of $\B(G)$ as formal $\Z$-linear combinations of symbols $[G/H]$.
Let $A$ be a field of characteristic $p\geq 0$. The representation ring $\R_A(G)$
of a finite group $G$ over $A$ is, as a group, the free abelian group on the set
of isomorphism classes of finitely generated indecomposable $A[G]$-modules (not
to be confused with the Grothendieck group of the category of finitely generated
$A[G]$-modules, which is also sometimes denoted by $\R_A(G)$). For every finite
group $G$ there is a natural homomorphism $\B(G)\rar \R_A(G)$, which sends the
isomorphism class represented by a $G$-set $X$ to the isomorphism class of the
$A[G]$-module $A[X]$ with a canonical $A$-basis given by the elements of $X$,
and with $G$ acting by permutations on this basis. Let $\K_A(G)$ denote the
kernel of this homomorphism. It is easy to see that $\K_A(G)$, as a subgroup of
$\B(G)$, only depends on the characteristic of $A$, and we refer to elements of
$\K_A(G)$ as \emph{Brauer relations of $G$ in characteristic $p$}.

It is an old problem, with many applications in number theory and geometry,
to understand the structure of $\K_A(G)$ for all finite groups $G$. See e.g.
\cite[\S 1]{bra1} for a brief overview of the history of the problem and of some
of the applications. The
most efficient and, from the point of view of number theoretic and geometric
applications, the most useful way of giving a complete characterisation of
$\K_A(G)$, not just as an abstract group, but with an explicit description of
generators, is to view $\K_A(G)$ as a Mackey functor with inflation. We briefly
explain informally what this means, and refer to Section \ref{sec:MFI} for
the formal discussion.

If $H$ is a subgroup of a finite group $G$, then Brauer relations of $H$ can be
induced to Brauer relations of $G$. Moreover, if $\bar{G}$ is a quotient of a
finite group $G$, then Brauer relations of $\bar{G}$ can be lifted to Brauer
relations of $G$. Let $\Imprim_{\K_A}(G)$ be the subgroup of $\K_A(G)$ generated
by all relations that are induced from proper subgroups or lifted from proper
quotients, and let $\Prim_{\K_A}(G)$ be the quotient $\K_A(G)/\Imprim_{\K_A}(G)$.
If one can give, for every finite group $G$, generators of $\Prim_{\K_A}(G)$,
then one obtains a list of Brauer relations with the property that all Brauer
relations in all finite groups are $\Z$-linear combinations of inductions and
lifts of relations in this list.

In \cite{bra1} the structure of $\Prim_{\K_A}(G)$ has been completely determined,
in the above sense, in the case when $A$ has characteristic 0. The following
theorem was a crucial step towards that result. If $q$ is a prime number, then
a group is called \emph{$q$-quasi-elementary} if it has a normal cyclic
subgroup of $q$-power index. A group is called \emph{quasi-elementary} if it is
$q$-quasi-elementary for some prime number $q$.

\begin{theorem}[\cite{bra1}, Theorem 4.3]\label{thm:main0}
Let $G$ be a finite group that is not quasi-elementary. Then:
\begin{enumerate}[leftmargin=*,label={\upshape(\alph*)}]
\item if all proper quotients of $G$ are cyclic, then $\Prim_{\K_{\Q}}(G)\cong \Z$;
\item if $q$ is a prime number such that all proper quotients of $G$ are $q$-quasi-elementary,
and at least one of them is not cyclic, then $\Prim_{\K_{\Q}}(G)\cong \Z/q\Z$;
\item if there exists a proper quotient of $G$ that is not quasi-elementary,
or if there exist distinct prime numbers $q_1$ and $q_2$ and, for $i=1$ and $2$,
a proper quotient of $G$ that is non-cyclic $q_i$-quasi-elementary, then
$\Prim_{\K_{\Q}}(G)$ is trivial.
\end{enumerate}
Moreover, in all cases, $\Prim_{\K_{\Q}}(G)$ is generated by any element of
$\K_{\Q}(G)\subseteq \B(G)$ of the form $[G/G]+\sum_{H\lneq G}a_H[G/H]$, $a_H\in \Z$.
\end{theorem}

Deligne \cite{Del-73} had proven a similar result on relations between monomial
representations, see Theorem \ref{thm:Del} below.

The main motivation for this paper is to understand $\Prim_{\K_A}(G)$ when $A$
has positive characteristic. To that end, we prove the following characteristic
$p$ analogue of Theorem \ref{thm:main0}, which will be used in a forthcoming
paper to give a characterisation of $\Prim_{\F_p}(G)$. If $p$ and $q$ are prime
numbers, then a group is called \emph{$p$-hypo-elementary} if it has a normal $p$-subgroup with
cyclic quotient, and it is called a \emph{$(p,q)$-Dress group} if it has a normal $p$-subgroup
with $q$-quasi-elementary quotient.

\begin{theorem}\label{thm:mainp}
Let $G$ be a finite group that is not a $(p,q)$-Dress group for any prime number
$q$. Then:
\begin{enumerate}[leftmargin=*,label={\upshape(\alph*)}]
\item if all proper quotients of $G$ are $p$-hypo-elementary,
then $\Prim_{\K_{\F_p}}(G)\cong \Z$;
\item if $q$ is a prime number such that all proper quotients of $G$ are
$(p,q)$-Dress groups,
and at least one of them is not $p$-hypo-elementary, then $\Prim_{\K_{\F_p}}(G)\cong \Z/q\Z$;
\item if there exists a proper quotient of $G$ that is not a $(p,q)$-Dress group
for any prime number $q$, or if there exist distinct prime numbers $q_1$ and $q_2$
and, for $i=1$ and $2$, a proper quotient of $G$ that is a non-$p$-hypo-elementary
$(p,q_i)$-Dress group, then $\Prim_{\K_{\F_p}}(G)$ is trivial.
\end{enumerate}
Moreover, in all cases, $\Prim_{\K_{F_p}}(G)$ is generated by any element of
$\K_{\F_p}(G)\subseteq \B(G)$ of the form $[G/G]+\sum_{H\lneq G}a_H[G/H]$, $a_H\in \Z$.
\end{theorem}

To prove part (b) of Theorem \ref{thm:mainp}, we prove an induction theorem
for $(p,q)$-Dress groups, which we believe to be of independent interest.
It is a characteristic $p$ analogue of the main Theorem of \cite{solomon}.
\begin{theorem}\label{thm:indthm}
Let $p$ and $q$ be prime numbers, let $G$ be a $(p,q)$-Dress group that is not
$p$-hypo-elementary, and let $a$
be an integer. Then there exists an element in $\K_{\F_p}(G)$ of the form
$a[G/G]+\sum_{H\lneq G}a_H[G/H]$, $a_H\in \Z$ if and only if $q|a$.
\end{theorem}

In fact, we deduce Theorems \ref{thm:main0} and \ref{thm:mainp}, as well as
Deligne's theorem on monomial relations, as special cases of a general result
on kernels of morphisms between Green functors with inflation. This formalism,
which is a mix of axiomatisations that have appeared in the literature many
times before, see e.g. \cite{Webb} and \cite{Boltje}, will be introduced
in Section \ref{sec:MFI}. In Section \ref{sec:primordial} we recall the concept
of primordial groups for a Mackey functor. Our main theorems on kernels of
morphisms of Green functors will be proven in Section \ref{sec:primitive}.
Section \ref{sec:appli} is devoted to concrete applications, and it is there
that we prove Theorems \ref{thm:main0}, \ref{thm:mainp}, and \ref{thm:indthm}.

\begin{acknowledgements}
During parts of this project, the first author was partially supported by
a Research Fellowship from the Royal Commission for the Exhibition
of 1851, and by an EPSRC First Grant, and the second author is supported by an
EPSRC Doctoral Grant. We would like to thank these institutions for their financial support.
We would also like to thank an anonymous referee for a careful reading of the paper
and for numerous helpful suggestions. 
\end{acknowledgements}

Our rings are always assumed to be associative, with a unit element. Let $R$ be a
commutative ring. By an $R$-algebra we mean a ring $A$ equipped with
a map $R\rar \rZ(A)$, where $\rZ(A)$ denotes the centre of $A$. If $\fp$ is
a prime ideal of $R$, then $R_{\fp}$ denotes the localisation of $R$ at $\fp$.
In this paper, $R$ will always denote a domain.

\section{Mackey and Green functors with inflation}\label{sec:MFI}
One can find many variations on the theme of Mackey functors in the literature.
The axiomatisation that we need is very similar to those of \cite{Webb,Boltje}.

\begin{definition}
A \emph{global Mackey functor with inflation} (MFI)
over $R$ is a collection $\cF$ of the following data.
\begin{itemize}
\item For every finite group $G$, $\cF(G)$ is an $R$-module;
\item for every monomorphism $\alpha\colon H\hookrightarrow G$ of finite groups,
$\cF_*(\alpha)\colon \cF(H)\rightarrow \cF(G)$ is a covariant $R$-module homomorphism
(which we think of as induction);
\item for every homomorphism $\epsilon\colon H\rar G$ of finite groups,
$\cF^{*}(\epsilon)\colon \cF(G)\rar \cF(H)$ is a contravariant $R$-module
homomorphism (which we think of as restriction when $\epsilon$ is a monomorphism,
and as inflation when $\epsilon$ is an epimorphism);
\end{itemize}
satisfying the following conditions.
\begin{enumerate}
\item[(MFI 1)] Transitivity of induction:
for all group monomorphisms
$U\stackrel{\beta}{\hookrightarrow}H\stackrel{\alpha}{\hookrightarrow}G$, we
have $\cF_*(\alpha\beta) = \cF_*(\alpha)\cF_*(\beta)$.
\item[(MFI 2)] Transitivity of restriction/inflation:
for all group homomorphisms
$U\stackrel{\beta}{\rar}H\stackrel{\alpha}{\rar}G$, we have
$\cF^*(\alpha\beta) = \cF^*(\beta)\cF^*(\alpha)$.
\item[(MFI 3)] For all inner automorphisms $\alpha\colon G\rar G$, we have
$\cF^*(\alpha)=\cF_*(\alpha)=~1$.
\item[(MFI 4)] For all automorphisms $\alpha$, we have
$\cF_*(\alpha)=\cF^*(\alpha^{-1})$.
\item[(MFI 5)] The Mackey condition: for all pairs of monomorphisms
$\alpha\colon H\hookrightarrow G$ and $\beta\colon K\hookrightarrow G$, we have
$$
\cF^*(\beta)\cF_*(\alpha)=\sum_{g\in \alpha(H)\backslash G/\beta(K)}\cF_*(\phi_g)\cF^*(\psi_g),
$$
where $\phi_g$ is the composition
$$
\phi_g\colon \beta(K)^g\cap \alpha(H)\stackrel{c_g}{\rar}
\beta(K)\cap {}^g\alpha(H)\hookrightarrow \beta(K)\stackrel{\beta^{-1}}{\rar}K,
$$
$c_g$ denoting conjugation by $g$,
and $\psi_g$ is the composition
$$
\psi_g\colon \alpha(H)\cap \beta(K)^g\hookrightarrow \alpha(H)
\stackrel{\alpha^{-1}}{\rar}H.
$$
\item[(MFI 6)] Commutativity of induction and inflation: whenever there is a commutative
diagram
$$
\xymatrix{
H\ar[d]_{\epsilon}\ar[r]^{\alpha}& G\ar[d]^{\delta}\\
\bar{H}\ar[r]^{\beta}& \bar{G},
}
$$
where $\epsilon,\delta$ are epimorphisms, and $\alpha,\beta$ are monomorphisms,
we have $\cF^{*}(\delta)\cF_*(\beta)=\cF_*(\alpha)\cF^{*}(\epsilon)$.
\end{enumerate}
\end{definition}

We will often use the following more intuitive notation: if $\cF$ is an MFI,
and $\alpha\colon H\hookrightarrow G$ is a monomorphism, we will write $\Res_{G/H}$ for
$\cF^*(\alpha)$, and $\Ind_{G/H}$ for $\cF_*(\alpha)$. The suppressed
dependence on $\alpha$ and $\cF$ will not cause any confusion. Similarly,
if $\epsilon\colon G\rar \bar{G}$ is an epimorphism with kernel $N$, we will
write $\Inf_{G/N}$ for $\cF^*(\epsilon)$.


\begin{definition}
A \emph{Green functor with inflation} (GFI) over $R$ is an MFI $\cF$ over $R$,
satisfying the following additional conditions.
\begin{enumerate}
\item[(GFI 1)] For every finite group $G$, $\cF(G)$ is an $R$-algebra.

\item[(GFI 2)] For every homomorphism $\alpha\colon H\rar G$ of finite groups,
$\cF^*(\alpha)$ is a homomorphism of $R$-algebras.

\item[(GFI 3)] Frobenius reciprocity: for every monomorphism $\alpha\colon H\hookrightarrow G$
and for all $x\in \cF(H)$, $y\in \cF(G)$, we have
\beq
\Ind_{G/H}(x)\cdot y & = &
\Ind_{G/H}(x\cdot \Res_{G/H}(y)),\\
y\cdot\Ind_{G/H}(x) & = &
\Ind_{G/H}(\Res_{G/H}(y)\cdot x).
\eeq
\end{enumerate}
\end{definition}
\begin{definition}
A \emph{morphism} from an MFI (respectively GFI) $\cF$ to an MFI (respectively GFI)
$\cG$ is a collection $r$ of $R$-module (respectively $R$-algebra) homomorphisms
$r_G: \cF(G)\rar \cG(G)$ for each finite group $G$, commuting in the obvious way
with $\cF_*,\cF^*,\cG_*,\cG^*$.
\end{definition}


\begin{definition}
Let $\cF$ be a GFI over $R$. A (left) \emph{module} under $\cF$ is an MFI
$\cM$ over $R$, satisfying the following conditions.
\begin{enumerate}
\item[(MOD 1)] For every group $G$, $\cM(G)$ is an $R$-linear (left) $\cF(G)$-module,
i.e. there is a map $\cF(G)\times \cM(G)\rar \cM(G)$ factoring through
$\cF(G)\otimes_R \cM(G)$.
\item[(MOD 2)] For every homomorphism $\epsilon\colon H\rar G$, and for all $x\in \cF(G)$,
$y\in \cM(G)$, we have
$$
\cM^*(\epsilon)(x\cdot y)=\cF^*(\epsilon)(x)\cdot \cM^*(\epsilon)(y).
$$

\item[(MOD 3)] For every monomorphism $\alpha\colon H\hookrightarrow G$ and for all
$x\in \cF(H)$, $y\in \cM(G)$, we have
$$
\cF_*(\alpha)(x)\cdot y = \cF_*(\alpha)(x\cdot\cM^*(\alpha)(y)).
$$
\end{enumerate}
\end{definition}
\begin{example}\label{ex:1}
The following are examples of GFIs over $\Z$.
\begin{enumerate}[leftmargin=*,label={\upshape(\alph*)}]
\item The Burnside ring functor $\B$: for a finite group $G$,
$\B(G)$ is the free abelian group on isomorphism classes $[X]$ of finite $G$-sets,
modulo the relations $[X\sqcup Y]-[X]-[Y]$ for all $G$-sets $X$, $Y$, and with
multiplication defined by $[X]\cdot [Y]=[X\times Y]$. Here, $\B_*$ is the usual
induction of $G$-sets, and $\B^*$ is inflation/restriction of $G$-sets.

\item\label{ex:repring} The representation ring functor $\R_F$ over a given field $F$:
for a finite group $G$, $\R_F(G)$ is the free abelian group on isomorphism
classes $[V]$ of finitely generated $F[G]$-modules, modulo the relations
$[U\oplus V]-[V]-[U]$, and with multiplication defined by
$[U]\cdot [V]=[U\otimes_F V]$, with diagonal $G$-action on the tensor product.
As in the previous example, $(\R_F)_*$ is induction of modules, and
$(\R_F)^*$ is inflation/restriction.

\item\label{ex:monomial} The monomial ring functor $\M$: for a finite group $G$,
$\M(G)$ is the free abelian group on conjugacy classes of symbols $[H,\lambda]$,
where $H$ runs over subgroups of $G$, and $\lambda$ runs over complex 1-dimensional
representations of $H$, and with multiplication defined by
\beq
\lefteqn{[H,\lambda]\cdot [K,\chi]=}\\
& & \sum_{g\in H\backslash G/K}[{}^gH\cap K,
\Res_{{}^gH/({}^gH\cap K)}{}^g\lambda\cdot \Res_{K/({}^gH\cap K)}\chi].
\eeq
If $\alpha:U\hookrightarrow G$ is a monomorphism, $[H,\lambda]\in \M(U)$,
and $[K,\chi]\in \M(G)$, then
\beq
& & \M_*(\alpha)([H,\lambda]) =[\alpha(H),\lambda\circ \alpha^{-1}],\\
& & \M^*(\alpha)([K,\chi]) =\\
& & \displaystyle \sum_{g\in \alpha(U)\backslash G/K}[\alpha^{-1}(\alpha(U)\cap {}^gK),\Res_{{}^gK/(\alpha(U)\cap {}^gK)} {}^g\chi\circ \alpha].
\eeq
\end{enumerate}
\end{example}

Every GFI is a module under itself, called the (left) \emph{regular module}.
We also have the obvious notions of sub-MFIs, sub-GFIs, and submodules.
\begin{definition}
A \emph{left ideal of
a GFI} is a sub-MFI that is also a submodule of the left regular module.
\end{definition}
\begin{definition}
Let $r:\cF\rar\cG$ be a morphism of MFIs over $\R$. Its \emph{kernel} $\cK$ is
defined as follows: for every finite group $G$, we define
$\cK(G)=\ker(r(G)\colon \cF(G)\rar\cG(G))$; for every homomorphism
$\epsilon\colon H\rar G$ of groups, we define $\cK^*(\epsilon)=\cF^*(\epsilon)|_{\cF(G)}$;
and for every monomorphism $\alpha\colon H\rar G$ of groups, we define
$\cK_*(\alpha)=\cF_*(\alpha)|_{\cF(H)}$. The \emph{image} of a morphism is
defined analogously. Let $\cF$ be a sub-MFI (respectively an ideal) of the MFI
(respectively GFI) $\cG$. The quotient $\cQ=\cG/\cF$ is defined as follows:
for every finite group $G$, we define 
$\cQ(G)=\cG(G)/\cF(G)$; for every homomorphism
$\epsilon\colon H\rar G$, we define $\cQ^*(\epsilon)=\cG^*(\epsilon) \pmod{\cF(H)}$;
and for every monomorphism $\alpha\colon H\rar G$, we define
$\cQ_*(\alpha)=\cG_*(\alpha) \pmod{\cF(G)}$.
\end{definition}

The proof of the following is routine and will be omitted.

\begin{lemma}\label{lem:routine}
\begin{enumerate}[leftmargin=*,label={\upshape(\alph*)}]
\item Let $r\colon \cF\rar\cG$ be a morphism of MFIs over $R$. Then its kernel is
a sub-MFI of $\cF$, and its image is a sub-MFI of $\cG$.
\item Let $r\colon \cF\rar\cG$ be a morphism of GFIs over $R$. Then its
kernel is an ideal of $\cF$, and its image is a sub-GFI of $\cG$.
\item Let $\cF$ be a sub-MFI of an MFI $\cG$. Then the quotient $\cG/\cF$
is an MFI.
\item Let $\cF$ be an ideal of a GFI $\cG$. Then $\cG/\cF$ is a GFI.
\end{enumerate}
\end{lemma}

\begin{example}\label{ex:2}
The following are some motivating examples for this work.

\begin{enumerate}[leftmargin=*,label={\upshape(\alph*)}]
\item\label{ex:kernel1} There is a GFI morphism $m'_{\C}\colon \M\rar \R_{\C}$, sending,
for every finite group $G$, a symbol $[H,\lambda]\in \M(G)$ to $\Ind_{G/H}\lambda\in \R_{\C}(G)$.
The kernel of $m'_{\C}$ was investigated by, among many others,
Langlands \cite{Lan-70}, Deligne \cite{Del-73}, Snaith \cite{Sna-88},
Boltje \cite{Boltje-EBI}, and Boltje--Snaith--Symonds \cite{BSS}.

\item\label{ex:kernel2} Let $F$ be a field. There is a GFI morphism
$m_F\colon \B\rar \R_F$, which maps, for every finite group $G$, a $G$-set $X$
to the permutation module $F[X]$ over $F$. Its kernel $\K_F$ is the MFI
of \emph{Brauer relations over $F$}. In \cite{bra1}, an explicit
description of generators of this MFI is given in the case when $F$ is a field
of characteristic 0. The primary motivation for this note
is to give a similarly explicit description when $F$ is a field of positive
characteristic.
\end{enumerate}
\end{example}

\section{Primordial groups}\label{sec:primordial}
If $S$ is a commutative $R$-algebra, and $\cF$ an MFI (respectively GFI) over $R$,
then $S\otimes_R \cF$, defined in the obvious way, is an MFI (respectively GFI)
over $S$. If $R=\Z$, then we will suppress any mention of $R$, and will just say
``$\cF$ is a MFI (respectively GFI)''. Throughout the rest of the paper, $Q$
will denote the field of fractions of $R$.
For a prime ideal $\mathfrak{p}$ of $R$, we will write
$\cF_{\mathfrak{p}}$ for $R_{\mathfrak{p}}\otimes_R \cF$, and $\cF_{Q}$ for
$Q \otimes_R \cF$.

\begin{notation}\label{not:indres}
Let $\cF$ be an MFI, and let $\cX$ be a class of groups
closed under isomorphisms. For every finite group $G$, we
define the following $R$-submodules of~$\cF(G)$:
\begin{eqnarray*}
\cI_{\cF,\cX}(G) & = & \sum_{H\leq G, \;H\in \cX}\Ind_{G/H}\cF(H),\\
\cI_{\cF}(G) & = & \sum_{H\lneq G}\Ind_{G/H}\cF(H),\\
\cK_{\cF,\cX}(G) & = & \bigcap_{H\leq G, \;H\in \cX} \ker(\Res_{G/H} \cF(G)),\\
\cK_{\cF}(G) & = & \bigcap_{H\lneq G} \ker(\Res_{G/H} \cF(G)).
\end{eqnarray*}
\end{notation}

\begin{definition}\label{def:primordial}
Let $\cF$ be an MFI and let $G$ be a finite group. We say that $G$ is
\emph{primordial} for $\cF$ if either $G$ is trivial, or
$\cF(G)\ne \cI_{\cF}(G)$. We denote the class
of all primordial groups for $\cF$ by $\cP(\cF)$.
\end{definition}
\begin{remark}\label{rem:IndThm}
Let $\cF$ be an MFI. 
\begin{enumerate}[leftmargin=*,label={\upshape(\alph*)}]
\item\label{item:inductionthm} 
Suppose that $\cX$ is a class of finite groups that is closed under isomorphisms
and under taking subgroups, with the property that for every finite group $G$, we
have $\cF(G)=\cI_{\cF,\cX}(G)$. Then it is shown in \cite[Theorem 2.1]{Thevenaz}
that $\cX$ contains the closure of $\cP(\cF)$ under taking all subgroups.

\item\label{item:1enough}
Suppose that $\cF$ is a GFI. Then it follows from axiom (GFI 3) that $G$ is
primordial for $\cF$ if and only if $1_{\cF(G)}\not\in\cI_{\cF}(G)$. It easily
follows from this and from axioms (GFI 2) and (MFI 6) that $\cP(\cF)$ is closed
under quotients.
\end{enumerate}
\end{remark}

\begin{example}\label{ex:prim}
\begin{enumerate}[leftmargin=*,label={\upshape(\alph*)}]
\item\label{item:primBurnside} Every finite group is primordial for the Burnside
ring functor $\B$, and also for $\B_{\Q}$.
Indeed, no non-zero multiple of the identity element of $\B(G)$ can be contained
in the image of induction from proper subgroups. Similarly,
every finite group is primordial for the monomial ring functor $\M$, and
also for $\M_{\Q}$.
\item Recall from Example \ref{ex:1} \ref{ex:repring} the representation
ring functor $\R_{\C}$. It follows from Brauer's induction theorem \cite[Theorem 5.6.4]{Benson}
that $\cP(\R_{\C})$ is contained in the class of elementary groups, i.e.
of direct products of finite cyclic groups by $p$-groups. Moreover, it is a
theorem of Green \cite{Green} that in fact $\cP(\R_{\C})$ consists precisely of
the elementary groups.
\item\label{item:elemprimord} Recall from Example \ref{ex:2} \ref{ex:kernel1} the GFI morphism
$m'_{\C}\colon\M\rar \R_\C$ from the monomial ring functor to the complex
representation ring functor. It follows from Brauer's induction theorem that
$(m'_{\C})_G$ is surjective for every finite group $G$, so by the previous example,
$\cP(\Im m'_{\C})$ consists precisely of the elementary groups.
\item\label{item:qelemprimord} Recall from Example \ref{ex:2} \ref{ex:kernel2}
the GFI morphism $m_{\Q}\colon \B\rar \R_{\Q}$. Let $q$ be a prime number.
Solomon's induction theorem implies that $\cP(\Im(m_\mathbb{Q})_q)$
is contained in the class of $q$-quasi-elementary groups, i.e. of
semidirect products $C\rtimes U$, with $C$ finite cyclic and $U$ a $q$-group.
Moreover, it is a theorem of Dokchitser \cite{solomon} that if $G$ is
$q$-quasi-elementary, then the trivial character of $G$ is not in the image of
induction of trivial characters from proper subgroups, so
$\cP(\Im(m_\Q)_q)$ is precisely the class of all $q$-quasi-elementary groups.
\item\label{item:cycprimord} Let $m_{\Q}$ be as above. It follows from Artin's
induction theorem \cite[Theorem 5.6.1]{Benson} that
$\cP(\Im(m_\Q)_{\Q})$ is the class of finite cyclic groups.
\item Let $p$ be a prime number, and let $m_{\F_p}\colon\B\rar \R_{\F_p}$ be
as in Example \ref{ex:2} \ref{ex:kernel2}. Dress's induction theorem
\cite[Theorem 9.4]{bra1} implies that $\cP(\Im m_{\F_p})$ is contained in the
class of all groups that are $(p,q)$-Dress groups for some prime number $q$. We
will show in Theorem \ref{thm:Dress} that the trivial representation of a
$(p,q)$-Dress group is not in the image of induction of trivial representations
from proper subgroups, so in fact, $\cP(\Im m_{\F_p})$ is precisely the class
of all finite groups that are $(p,q)$-Dress groups for some prime number $q$.
\end{enumerate}
\end{example}

\section{The primitive quotient}\label{sec:primitive}
In this section, we prove our main theorems on kernels of morphisms of GFIs.
The main results of the section are Theorem \ref{thm:noprim}, \ref{thm:whichG},
and \ref{thm:primeqR}.
\begin{lemma}\label{lem:MeqI1}
Let $m\colon \cF\rar \cG$ be a morphism of GFIs over a ring $R$ with kernel
$\cK$, and let $G$ be a finite group. Then the following are equivalent:
\begin{enumerate}[leftmargin=*,label={\upshape(\roman*)}]
\item the group $G$ is not primordial for $\Im m$;
\item\label{item:x} for each proper subgroup $H$ of $G$,
there exists an element $x_H\in \cF(H)$ such that
$x=1_{\cF(G)}+ \sum_{H \lneq G} \Ind_{G/H}(x_H)\in\cK(G)$.
\end{enumerate}
\end{lemma}
\begin{proof}
By Remark \ref{rem:IndThm} \ref{item:1enough}, $G$ is not primordial
for $\Im m$ if and only if
$$
m_G(1_{\cF(G)})\in
\sum_{H\lneq G}\Ind_{G/H}(m_H(\cF(H)))=m_G \sum_{H\lneq G} \Ind_{G/H}(\cF(H)).
$$
This is equivalent to the existence of elements $x_H\in \cF(H)$ for $H\lneq G$
such that
$x=1_{\cF(G)}+ \sum_{H \lneq G} \Ind_{G/H}(x_H)\in \cK(G)$.
\end{proof}

\begin{definition}
Let $G$ be a finite group, let $\cF$ be a GFI over $R$, and let $\cM$ be a module
under $\cF$. Let $D(G)$ be an $R$-subalgebra of the centre of $\cF(G)$. Define
the set of \emph{$D$-imprimitive elements} of $\cM(G)$ by
$$
\Imprim_{\cM,D}(G)=D(G)\cdot \left(\sum_{H\lneq G}\Ind_{G/H}\cM(H)+
\sum_{1\neq N\normal G}\Inf_{G/N}\cM(G/N)\right).
$$
This is an $R$-submodule of $\cM(G)$. Define the \emph{$D$-primitive quotient}
of $\cM(G)$ to be the quotient of $R$-modules
$$
\Prim_{\cM,D}(G) = \cM(G)/\Imprim_{\cM,D}(G).
$$
When $D(G)$ is generated by $1_{\cF(G)}$ over $R$, we will drop it from the notation.
\end{definition}

\begin{notation}\label{not:4}
For the rest of the section, we put ourselves in the following situation.
We fix a morphism $m: \cF \rar \cG$ of GFIs over a domain $R$ with the
property that $\cF(H)$ is $R$-torsion free for all finite groups $H$,
and we let $\cK$ denote its kernel. Recall from Lemma \ref{lem:routine}
that $\cK$ is an ideal of $\cF$. Further, we fix a finite group $G$, and an
$R$-subalgebra $D(G)$ of the centre of $\cF(G)$. Assume for the rest of the
section that the $R$-module $\cF(G)$ is generated by $\cI_{\cF}(G)$ and $D(G)$.
\end{notation}

\begin{lemma}\label{lem:ideal}
Under the hypotheses of Notation \ref{not:4}, let $\cM$ be any module under
$\cF$, and let $x$ be any element of $\cM(G)$. Then the $R$-submodule of
$\cM(G)$ generated by $D(G)\cdot \cI_{\cM}(G)$ and $D(G)\cdot x$ is an
$\cF(G)$-submodule.
\end{lemma}
\begin{proof}
Let $\Theta$ be an element of the $R$-module $D(G)\cdot \cI_{\cM}(G) + D(G)\cdot x$,
and let $\alpha\in \cF(G)$.
If $\alpha=\Ind_{G/H}y$ for some $y \in \cF(H)$, where $H$ is a proper subgroup
of $G$, then by property (MOD 3), $\alpha\cdot \Theta=\Ind_{G/H}(y\cdot
\Res_{G/H}\Theta)\in \cI_{\cM}(G)$. If, on the other hand, $\alpha\in D(G)$,
then $\alpha\cdot \Theta\in D(G)\cdot \cI_{\cM}(G)+D(G)\cdot x$ by definition.
Since $\cF(G)$ is assumed to be generated by $\cI_{\cF}(G)$ and
by $D(G)$, it follows that $\alpha\cdot \Theta\in D(G)\cdot\cI_{\cM}(G)+D(G)\cdot x$
for all $\alpha\in \cF(G)$.
\end{proof}

\begin{lemma}\label{lem:MeqI2}
Under the hypotheses of Notation \ref{not:4},
suppose that the equivalent conditions of Lemma \ref{lem:MeqI1}
are satisfied for $m$ and $G$, and let $x\in\cK(G)$ be an element of the form
$x=1_{\cF(G)}+ \sum_{H \lneq G} \Ind_{G/H}(x_H)$, where $x_H\in \cF(H)$. Then
$$
\cK(G)=D(G)\cdot \cI_{\cK}(G) + D(G) \cdot x.
$$
\end{lemma}
\begin{proof}
Let $I=D(G)\cdot \cI_{\cK}(G) + D(G) \cdot x\subseteq \cK(G)$.
We claim that $\cK(G)\subseteq I$.
Let $y\in \cK(G)$. Lemma \ref{lem:ideal} implies that $I$ is an ideal of
$\cF(G)$. Since we have $x\in D(G)\cdot x\subseteq I$, it follows that
$y\cdot x\in I$. Also,
$$
y\cdot x-y = \sum_{H \lneq G} y\cdot \Ind_{G/H}(x_H)=
\sum_{H \lneq G} \Ind_{G/H}(\Res_{G/H}(y) \cdot x_H)
$$
is in $\cI_{\cK}(G)$, and therefore in $I$. It follows that
$y=y\cdot x +(y-y\cdot x)\in I$. Thus $\cK(G)\subseteq I$, and the proof is complete.
\end{proof}

\begin{theorem}\label{thm:noprim}
Under the hypotheses of Notation \ref{not:4},
suppose that there is a non-trivial normal subgroup $N$ of $G$ such that
$G/N$ is not primordial for $\Im m$. Then $\Prim_{\cK,D}(G)$ is trivial.
\end{theorem}
\begin{proof}
By Lemma \ref{lem:MeqI1}, applied to the quotient $G/N$,
there exists an element
$z=1_{\cF(G/N)}+ \sum_{H/N \lneq G/N} \Ind_{(G/N)/(H/N)}(x_H)\in \cK(G/N)$. Since
$N$ is non-trivial, the inflation $x=\Inf_{G/N}z$ is contained in $\Imprim_{\cK,D}(G)$.
It follows from Lemma \ref{lem:MeqI2} that
$\cK(G)=D(G)\cdot \cI_{\cK}(G) + D(G) \cdot x
\subseteq\Imprim_{\cK,D}(G)$, as claimed.
\end{proof}
\begin{theorem}\label{thm:whichG}
Under the hypotheses of Notation \ref{not:4},
suppose that $G$ is non-trivial, and that $\Prim_{\cK,D}(G)$ is non-trivial.
Then $G$ is an extension of the
form $1\rightarrow S^d\rightarrow G \rightarrow H \rightarrow1$, where $S$ is a
finite simple group, and $H$ is primordial for $\Im m$.
\end{theorem}
\begin{proof}
By the existence of chief series, there exists a normal subgroup of $G$ that
is isomorphic to $S^d$, where $S$ is a finite simple group, and $d\geq 1$ is an
integer. By Theorem \ref{thm:noprim}, the quotient $G/S^d$ is primordial for $\Im m$.
\end{proof}
\begin{assumption}\label{ass:4}
In addition to the assumptions of Notation \ref{not:4}, we now assume that:
\begin{itemize}[leftmargin=*]
\item the ring $R$ is a Euclidean domain;
\item for every normal subgroup $N$ of $G$, the inflation map
$\Inf_{G/N}\colon \cF(G/N)\rar \cF(G)$ is injective;
\item\label{item:gen}
for every quotient $G/N$, the $R$-module $\cF(G/N)$ is generated by $\cI_{\cF}(G/N)$
and $1$. In particular, the subalgebra $D(G)$ will be assumed to be generated by
$1_{\cF(G)}$ over $R$, and will now be dropped from the notation.
\end{itemize}
\end{assumption}

\begin{theorem} \label{thm:primeqR}
Under the hypotheses of Notation \ref{not:4} and Assumption \ref{ass:4}, suppose
that $G$ is primordial for $\cF_Q$ and not primordial for $\Im m$. Let $\fa$
be the ideal of $R$ generated by all those $a\in R$ for which there exists
a proper quotient $G/N$ and an element $a1_{\cF(G/N)}+y\in \cK(G/N)$ with
$y\in \cI_{\cF}(G/N)$. Then $\Prim_{\cK}(G)$ is isomorphic to $R/\fa$
and is generated by the image of any element of the form
$x=1_{\cF(G)}+\sum_{H\lneq G} \Ind_{G/H}x_H\in\cK(G)$.
\end{theorem}
\begin{proof}
By Lemma \ref{lem:MeqI2}, the quotient $\Prim_{\cK}(G)$ is generated by any
$x\in \cK(G)$ of the form $x=1_{\cF(G)}+\sum_{H\lneq G}\Ind_{G/H}x_H$, where
$x_H\in \cF(H)$. Since by assumption $G$ is primordial for $\cF_Q$, Remark
\ref{rem:IndThm} \ref{item:1enough} implies that $ax\not\in \cI_{\cK}(G)$ for
any non-zero $a\in R$. It also follows from the same remark and from the
assumptions \ref{not:4} and \ref{ass:4} that any element of $\cK(G)$ can be
uniquely written as $a1_{\cF(G)}+y$, where $a\in R$ and $y\in \cI_{\cF}(G)$,
and analogously for any element of $\cK(G/N)$ for every normal subgroup $N$ of
$G$. We deduce that the annihilator $\fa\subseteq R$ of $x+\Imprim_{\cK}(G)\in
\Prim_{\cK}(G)$ is generated, as an $R$-module, by all those $a\in R$ for
which there exists a non-trivial normal subgroup $N$ of $G$ and an element
$a1_{\cF(G/N)}+y\in \cK(G/N)$, where $y\in \cI_{\cF}(G/N)$. Moreover,
we then have $\Prim_{\cK}(G)\cong R/\fa$, as claimed.
\end{proof}

\begin{corollary}\label{cor:primR}
Under the hypotheses of Theorem \ref{thm:primeqR}, if all proper quotients
of $G$ are primordial for $(\Im m)_Q$, then $\Prim_{\cK}(G)$ is isomorphic
to~$R$.
\end{corollary}
\begin{proof}
Since all proper quotients $G/N$ are primordial for $(\Im m)_Q$, Remark
\ref{rem:IndThm} \ref{item:1enough} implies that the ideal $\fa$ of Theorem
\ref{thm:primeqR} is zero.
\end{proof}

\begin{corollary}\label{cor:primCp}
Under the hypotheses of Theorem \ref{thm:primeqR}, suppose that there exists a
prime ideal $\fp$ of $R$ such that for every prime ideal $\fq\neq \fp$
there exists a proper quotient of $G$ that is not primordial for $(\Im m)_{\fq}$.
Then $\Prim_{\cK}(G)\cong R/\fp^n$, where $n$ is the smallest non-negative
integer for which there exists a proper quotient $G/N$ and an element
$a1_{\cF(G/N)}+y\in \cK(G/N)$ with $a\in \fp^n\setminus\{0\}$ and $\Inf_{G/N}y\in \cI_{\cF}(G)$.
\end{corollary}
\begin{proof}
Let $\fq\neq \fp$ be a prime ideal of $R$.
By Lemma \ref{lem:MeqI1}, applied to the map $\cF_{\fq}\rar \cG_{\fq}$
and to a proper quotient $G/N\not\in \cP((\Im m)_{\fq})$, there exists
$a\in \fa$ that is not in $\fq$, where $\fa$ is the ideal of Theorem \ref{thm:primeqR}.
Since $R$ is a Euclidean domain, this implies that $\fa=\fp^n$ for some
integer $n\geq 0$.
\end{proof}

\begin{corollary}\label{cor:primtriv}
Under the hypotheses of Theorem \ref{thm:primeqR}, suppose that
for every non-zero prime ideal $\fp$ of $R$ there exists a proper quotient
of $G$ that is not primordial for $(\Im m)_{\fp}$. Then $\Prim_{\cK}(G)$ is
trivial.
\end{corollary}
\begin{proof}
Let $\fp$ be any non-zero prime ideal. By Lemma \ref{lem:MeqI1}, applied to
the map $\cF_{\fp}\rar \cG_{\fp}$
and to a proper quotient $G/N\not\in \cP((\Im m)_{\fp})$,
there exists $a\in \fa$ that is not in $\fp$, where $\fa$
is the ideal of Theorem \ref{thm:primeqR}. Since $R$ is
a Euclidean domain, it follows that $1\in \fa$.
\end{proof}

\section{Applications}\label{sec:appli}

In this section we explicate the results of Section \ref{sec:primitive} in the
case of monomial relations and of Brauer relations. The main new results are
on Brauer relations in positive characteristic, but we also show how to derive
some known results on monomial relations and on Brauer relations
in characteristic 0 from the formalism of GFIs. In particular,
we prove Theorems \ref{thm:main0}, \ref{thm:mainp}, and \ref{thm:indthm}
from the introduction. The following result, although not explicitly stated,
is proved in \cite{Del-73} along the way to a complete classification of monomial
relations in soluble groups.

\begin{theorem}[Deligne--Langlands, \cite{Del-73}]\label{thm:Del}
Let $\K'_{\bC}$ be the kernel of the morphism of GFIs $m_{\bC}'\colon \M\rar \R_{\bC}$
as in Example \ref{ex:2} \ref{ex:kernel1}. Let $G$ be a finite group that has
a non-trivial normal subgroup $N$ such that $G/N$ is not elementary. Let
$D(G)$ be generated over $\Z$ by symbols $[G,\lambda]$, as $\lambda$ runs over
isomorphism classes of 1-dimensional representations of $G$. Then
$\Prim_{\K'_{\bC},D}(G)$ is trivial.
\end{theorem}
\begin{proof}
By Example \ref{ex:prim} \ref{item:elemprimord}, the primordial groups for
$\Im m_{\bC}'$ are precisely the elementary groups, and every group is primordial
for $\M_{\Q}$ (see Example \ref{ex:prim} \ref{item:primBurnside}). It easily follows
that the assumptions of Notation \ref{not:4} are satisfied for this morphism of GFIs
and this choice of $D(G)$. The result therefore follows from Theorem \ref{thm:noprim}.
\end{proof}

\begin{theorem}[Bartel--Dokchitser, \cite{bra1}]
Let $\K_{\Q}$ be the kernel of the morphism of GFIs $m_{\bQ}\colon\B\rar \R_{\bQ}$
as in Example \ref{ex:2} \ref{ex:kernel2}, and let $G$ be a finite group that is
not quasi-elementary. Then:
\begin{enumerate}[leftmargin=*,label={\upshape(\alph*)}]
\item if all proper quotients of $G$ are cyclic, then $\Prim_{\K_{\Q}}(G)\cong \Z$;
\item if $q$ is a prime number such that all proper quotients of $G$ are $q$-quasi-elementary,
and at least one of them is not cyclic, then $\Prim_{\K_{\Q}}(G)\cong \Z/q\Z$;
\item if there exists a proper quotient of $G$ that is not quasi-elementary,
or if there exist distinct prime numbers $q_1$ and $q_2$ and, for $i=1$ and $2$,
a proper quotient of $G$ that is non-cyclic $q_i$-quasi-elementary, then
$\Prim_{\K_{\Q}}(G)$ is trivial.
\end{enumerate}
Moreover, in all cases, $\Prim_{\K_{\Q}}(G)$ is generated by any element of
$\K_{\Q}(G)\subseteq \B(G)$ of the form $[G/G]+\sum_{H\lneq G}a_H[G/H]$, $a_H\in \Z$.
\end{theorem}
\begin{proof}
By Example \ref{ex:prim} \ref{item:cycprimord}, $\cP((\Im m_{\Q})_{\Q})$ is the
class of cyclic groups. Let $q$ be a prime number. By Example \ref{ex:prim}
\ref{item:qelemprimord}, $\cP(\Im m_{\Q})$ is the class of quasi-elementary groups, and
$\cP((\Im m_{\Q})_q)$ is the class of $q$-quasi-elementary groups. Moreover, if
$U$ is a non-cyclic $q$-quasi-elementary group, then by \cite{solomon}, there exists
an element of $\K_{\Q}(U)\subseteq \B(U)$ of the form $q[U/U]+\sum_{H\lneq U}a_H[U/H]$.
Since every finite group is primordial for $\B$, the hypotheses of Theorem \ref{thm:primeqR}
are satisfied. Part (a) of the theorem therefore follows from Corollary
\ref{cor:primR}. Finally, note that if $q_1$ and $q_2$ are distinct prime numbers,
then a finite group is both $q_1$-quasi-elementary and $q_2$-quasi-elementary if and only
if it is cyclic. Parts (b) and (c) of the theorem therefore follow from Corollaries
\ref{cor:primCp} and \ref{cor:primtriv}, respectively.
\end{proof}

Fix a prime number $p$. The rest of the section is devoted to the kernel $\K_{\F_p}$
of the morphism of GFIs $m_{\F_p}\colon \B\rar \R_{\F_p}$ as in Example 
\ref{ex:2} \ref{ex:kernel2}. 

First, we prove Theorem \ref{thm:indthm}, which is a characteristic $p$ analogue
of the main result of \cite{solomon}. We recall the statement.

\begin{theorem}\label{thm:Dress}
Let $q$ be a prime number, let $G$ be a $(p,q)$-Dress group that is not
$p$-hypo-elementary, and let $a$ be an
integer. Then $a[G/G]\in \cI_{\Im m_{\F_p}}(G)$ if and only if $q|a$.
\end{theorem}
\begin{proof}
Since $G$ is a $(p,q)$-Dress group, it is an extension of a $q$-group $U$ by a
normal $p$-hypo-elementary subgroup $N=P\rtimes C$, where $P$ is a $p$-group and
$C$ is cyclic of order coprime to $pq$.

First we prove that if $a[G/G]\in\cI_{\Im m_{\F_p}}(G)$, then $q|a$. Suppose
that there exist integers $a_H$ for $H\lneq G$ such that
$$
a\F_p[G/G]=\sum_{H\lneq G}a_H\F_p[G/H]\in \R_{\F_p}(G),
$$
where the sum runs over representatives of conjugacy classes of subgroups of
$G$, and where $\F_p[G/H]\in \R_{\F_p}(G)$ denotes the linear permutation module
$\Ind_{G/H}\triv_H$ over $\F_p$. By restricting to the normal $p$-hypo-elementary
subgroup $N$, we find that
\begin{eqnarray}\label{eq:Dress}
a\F_p[N/N]=\sum_{H\lneq G}a_H\sum_{g\in G/HN}\F_p[N/N\cap gHg^{-1}].
\end{eqnarray}
By Conlon's Induction Theorem \cite[Lemma 81.2]{CR2}, $p$-hypo-elementary
groups are primordial for $\Im m_{\F_p}$, so the coefficient of $\F_p[N/N]$ on
the right hand side of equation \ref{eq:Dress} must be equal to $a$:
$$
a=\sum_{N\leq H\lneq G}a_H\cdot \#(G/H).
$$
But for every $H\leq G$ that contains $N$, the quantity $\#(G/H)$ is divisible
by $q$, so $a$ is divisible by $q$, as claimed.

Now we show that $q[G/G]\in \cI_{\Im m_{\F_p}}(G)$.
First, we treat a special case:
assume that $P$ is the trivial group, so that $G\cong C\rtimes U$ is
non-cyclic $q$-quasi-elementary, where $C$ is cyclic of order coprime to $pq$.
Assume further that either $p\neq q$, or $U$ acts faithfully on $C$. 
By \cite{solomon}, there exists an element $x=q[G/G]+\sum_{H\lneq G}a_H[G/H]\in \K_{\Q}(G)$.
By Artin's Induction Theorem \cite[Theorem 5.6.1]{Benson}, this is equivalent
to the statement that there
exists an $x\in \K_{\Q}(G)$ as above such that for all cyclic subgroups
$H\leq G$, we have $f_H(x)=0$, where $f_H\colon \B(G)\rar \Z$
is defined on a $G$-set $X$ as the number of fixed points $\#X^H$.
But under the hypotheses on $G$, the cyclic subgroups of $G$ are precisely
the $p$-hypo-elementary subgroups of $G$. By Conlon's Induction Theorem \cite[Lemma 81.2]{CR2},
the above statements are therefore equivalent to the existence of an element
$x=q[G/G]+\sum_{H\lneq G}a_H[G/H]\in \K_{{\F_p}}(G)$, as required.

Now, we deduce the general case. Given a non-$p$-hypo-elementary $(p,q)$-Dress
group $G$, let $\tilde{G}=G/P$. This is a non-cyclic $q$-quasi-elementary group,
$\tilde{G}=C\rtimes U$, where $U$ is a $q$-group, and $C$ is cyclic of order
coprime to $pq$. Let $K$ be the kernel of the action of $U$ on $C$. If $K=U$
and $p=q$, then $\tilde{G}\cong C\times U$, and $G$ is $p$-hypo-elementary,
contradicting the assumptions.
Otherwise, $\bar{G}=\tilde{G}/K$ is as in the special case above, so there
exists an element $x=q[\bar{G}/\bar{G}]+\sum_{H\lneq \bar{G}}a_H[\bar{G}/H]\in \K_{\F_p}(\bar{G})$.
Taking the inflation of $x$ to $G$ yields the desired element of $\K_{\F_p}(G)$,
and the proof is complete.
\end{proof}

\begin{corollary}\label{cor:Dress}
Let $q$ be a prime number. Then $\cP((\Im m_{\F_p})_q)$ is the class of 
$(p,q)$-Dress groups.
\end{corollary}
\begin{proof}
By Dress's Induction Theorem in the version as stated in \cite[Theorem 9.4]{bra1},
and by Remark \ref{rem:IndThm} \ref{item:inductionthm}, all primordial groups
for $(\Im m_{\F_p})_q$ are $(p,q)$-Dress groups. The reverse inclusion
follows from Theorem \ref{thm:Dress}.
\end{proof}

\begin{theorem}
Let $G$ be a finite group that is not a $(p,q)$-Dress group for any prime number
$q$. Then:
\begin{enumerate}[leftmargin=*,label={\upshape(\alph*)}]
\item if all proper quotients of $G$ are $p$-hypo-elementary,
then $\Prim_{\K_{\F_p}}(G)\cong \Z$;
\item if $q$ is a prime number such that all proper quotients of $G$ are
$(p,q)$-Dress groups,
and at least one of them is not $p$-hypo-elementary, then $\Prim_{\K_{\F_p}}(G)\cong \Z/q\Z$;
\item if there exists a proper quotient of $G$ that is not a $(p,q)$-Dress group
for any prime number $q$, or if there exist distinct prime numbers $q_1$ and $q_2$
and, for $i=1$ and $2$, a proper quotient of $G$ that is a non-$p$-hypo-elementary
$(p,q_i)$-Dress group, then $\Prim_{\K_{\F_p}}(G)$ is trivial.
\end{enumerate}
Moreover, in all cases, $\Prim_{\K_{F_p}}(G)$ is generated by any element of
$\K_{\F_p}(G)\subseteq \B(G)$ of the form $[G/G]+\sum_{H\lneq G}a_H[G/H]$, $a_H\in \Z$.
\end{theorem}
\begin{proof}
By Conlon's Induction Theorem \cite[Lemma 81.2]{CR2}, $\cP((\Im m_{\F_p})_{\Q})$
is the class of $p$-hypo-elementary groups. Let $q$ be a prime number. By
Corollary \ref{cor:Dress}, $\cP((\Im m_{\F_p})_q)$ is the class of $(p,q)$-Dress
groups, and $\cP(\Im m_{\F_p})$ is the class of all groups that are $(p,q')$-Dress
groups for some prime number $q'$. Moreover, if $U$ is a non-$p$-hypo-elementary
$(p,q)$-Dress group, then by Theorem \ref{thm:Dress}, there exists an element of
$\K_{\F_p}(U)\subseteq \B(U)$ of the form $q[U/U]+\sum_{H\lneq U}a_H[U/H]$.
Part (a) of the theorem follows from Corollary \ref{cor:primR}. Finally, note
that if $q_1$ and $q_2$ are distinct prime numbers, then a finite group is
both a $(p,q_1)$-Dress group and a $(p,q_2)$-Dress group if and only if it is
$p$-hypo-elementary.
Parts (b) and (c) of the theorem therefore follow from
Corollaries \ref{cor:primCp} and \ref{cor:primtriv}, respectively.
\end{proof}

\end{document}